\documentclass[reqno]{amsart}

\usepackage{amsfonts}
\usepackage{amsmath}
\usepackage{amssymb}
\usepackage{latexsym}
\usepackage{amscd}
\usepackage{amsthm}
\usepackage{mathrsfs}
\usepackage{subfigure}
\usepackage{graphicx}
\usepackage[all]{xy}
\usepackage{color}

\newtheorem{thm}{Theorem}[section]

\newtheorem{lem}[thm]{Lemma}
\newtheorem{prob}[thm]{Problem}
\theoremstyle{definition}

\newtheorem{rem}[thm]{Remark}

\newcommand{\M}{\mathcal{M}}
\newcommand{\T}{\mathcal{T}}

\newcommand{\Z}{\mathbb{Z}}

\author[R. Kobayashi]{Ryoma Kobayashi}
\address[R. Kobayashi]{
Department of General Education,\endgraf
National Institute of Technology, Ishikawa College,\endgraf
Tsubata, Ishikawa, 929-0392, Japan
}
\email{kobayashi\_ryoma@ishikawa-nct.ac.jp}


\thanks{The author was supported by JSPS KAKENHI Grant Number JP26K06797.
}

\begin{document}

\title[Generating $\T_2(N_g)$]{Generating the twist subgroup of the level $2$ mapping class group of a non-orientable closed surface}

\maketitle

\begin{abstract}
We give two small generating sets for the subgroup of the mapping class group of a non-orientable closed surface which is generated by Dehn twists and acts trivially on the $\Z_2$ coefficient first homology group of the surface.
\end{abstract}

\section{Introduction}

Let $N_g$ be a genus $g\ge1$ non-orientable closed surface, that is, $N_g$ is a connected sum of $g$ real projective planes.
In this paper, we describe $N_g$ as a surface which is obtained by attaching $g$ M\"obius bands to the boundaries of a sphere with $g$ boundaries, as shown in Figure~\ref{non-ori-surf}.
The \textit{mapping class group} $\M(N_g)$ of $N_g$ is the group consisting of isotopy classes of self-diffeomorphisms of $N_g$.
The \textit{level $2$ mapping class group} $\M_2(N_g)$ of $N_g$ is the subgroup of $\M(N_g)$ acting trivially on $H_1(N_g;\Z_2)$.
The \textit{twist subgroup} $\T(N_g)$ of $\M(N_g)$ is the subgroup of $\M(N_g)$ generated by Dehn twists which are defined the next paragraph.
Let $\T_2(N_g)=\M_2(N_g)\cap\T(N_g)$.
We call $\T_2(N_g)$ the twist subgroup of $\M_2(N_g)$ or the level $2$ twist subgroup of $\M(N_g)$.
In this paper, we give two small generating sets for $\T_2(N_g)$.

\begin{figure}[htbp]
\includegraphics{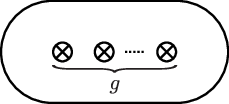}
\caption{A model of a non-orientable closed surface $N_g$.
The marks $\otimes$ are attached M\"obius bands and called \textit{crosscaps}.}\label{non-ori-surf}
\end{figure}

First, we define a \textit{Dehn twist} and a \textit{crosscap slide} which are elements of $\M(N_g)$.
For a simple closed curve $\alpha$ of $N_g$, its regular neighborhood is either an annulus or a M\"obius band.
We call $\alpha$ a \textit{two sided} or a \textit{one sided} simple closed curve, respectively.
For a two sided simple closed curve $\alpha$, the \textit{Dehn twist} $t_\alpha$ about $\alpha$ is the isotopy class of the map acting as shown in Figure~\ref{dehn-slide}~(a).
The direction of the twist is indicated by an arrow written beside $\alpha$ as shown in Figure~\ref{dehn-slide}~(a).
For a one sided simple closed curve $\mu$ of $N_g$ and an oriented two sided simple closed curve $\alpha$ of $N_g$ such that $N_g\setminus\alpha$ is non-orientable when $g\geq3$ and that $\mu$ and $\alpha$ intersect transversely at only one point, the \textit{crosscap slide} $Y_{\mu,\alpha}$ about $\mu$ and $\alpha$ is the isotopy class of the map described by pushing the crosscap which is a regular neighborhood of $\mu$ once along $\alpha$, as shown in Figure~\ref{dehn-slide}~(b).
Note that $\M(N_g)$ can be generated by Dehn twists and crosscap slides (see~\cite{Li1}), and can not be generated by either Dehn twists or crosscap slides (see~\cite{Li2}).

\begin{figure}[htbp]
\subfigure[The Dehn twist $t_\alpha$ about $\alpha$.]{\includegraphics{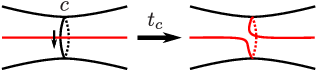}}~\subfigure[The crosscap slide $Y_{\mu,\alpha}$ about $\mu$ and $\alpha$.]{\includegraphics{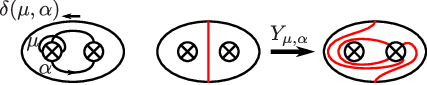}}
\caption{Descriptions of a Dehn twist and a crosscap slide.}\label{dehn-slide}
\end{figure}

Next, we explain about some relations on Dehn twists and crosscap slides.
For details, for instance see \cite{Sz1}.
For a simple closed curve $\alpha$ which bounds a disk or a crosscap, we have that
$$t_\alpha=1.$$
For any Dehn twist $t_\alpha$, crosscap slide $Y_{\mu,\alpha}$ and $f\in\M(N_g)$, we have that
$$ft_\alpha{}f^{-1}=t_{f(\alpha)},~fY_{\mu,\alpha}f^{-1}=Y_{f(\mu),f(\alpha)},$$
where the directions of the twist of $t_{f(\alpha)}$ and the pushing of $Y_{f(\mu),f(\alpha)}$ are induced from $f$ and the directions of the twist of $t_\alpha$ and the pushing of the $Y_{\mu,\alpha}$, respectively.
Let $\mu_1$, $\mu_2$ and $\alpha$ be simple closed curves as shown in Figure~\ref{crosscap-slide-rel}~(a).
Then we have that
$$Y_{\mu_2,\alpha}^{-1}Y_{\mu_1,\alpha}=Y_{\mu_2,\alpha}Y_{\mu_1,\alpha}^{-1}=t_\alpha^2.$$
For two crosscap slides $Y_{\mu,\alpha}$ and $Y_{\mu,\beta}$ such that $\overline{\alpha}\overline{\beta}$ is simple, where $\overline{\alpha}$ and $\overline{\beta}$ are oriented loops of $N_{g-1}$ which is obtained by collapsing a regular neighborhood of $\mu$ and $\overline{\alpha}\overline{\beta}$ is a composition loop of $\overline{\alpha}$ and $\overline{\beta}$ based at the collapsing point, let $\delta_1(\mu,\alpha,\beta)$ and $\delta_2(\mu,\alpha,\beta)$ be simple closed curves determined by $\mu$, $\alpha$ and $\beta$ as shown in Figure~\ref{crosscap-slide-rel}~(b).
Then we have that
$$Y_{\mu,\beta}Y_{\mu,\alpha}=t_{\delta_1(\mu,\alpha,\beta)}t_{\delta_2(\mu,\alpha,\beta)}.$$
In addition, for any crosscap slide $Y_{\mu,\alpha}$, we have that
$$Y_{\mu,\alpha}^2=t_{\delta(\mu,\alpha)},~Y_{\mu,\alpha^{-1}}=Y_{\mu,\alpha}^{-1},$$
where $\delta(\mu,\alpha)$ is a simple closed curve determined by $\mu$ and $\alpha$ as shown in Figure~\ref{dehn-slide}~(b) and $\alpha^{-1}$ is the simple closed curve whose orientation is inverse of $\alpha$.

\begin{figure}[htbp]
\subfigure[$Y_{\mu_2,\alpha}^{-1}Y_{\mu_1,\alpha}=Y_{\mu_2,\alpha}Y_{\mu_1,\alpha}^{-1}=t_\alpha^2$]{\begin{minipage}[t]{0.45\textwidth}\centering\includegraphics{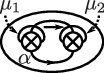}\end{minipage}}
\subfigure[$Y_{\mu,\beta}Y_{\mu,\alpha}=t_{\delta_1(\mu,\alpha,\beta)}t_{\delta_2(\mu,\alpha,\beta)}$]{\includegraphics{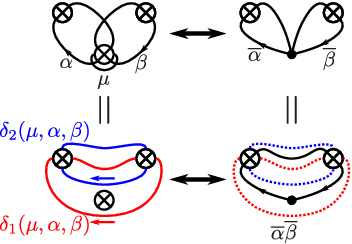}}
\caption{Relations on crosscap slides.}\label{crosscap-slide-rel}
\end{figure}

We now explain about background on our works and main results.
Szepietowski~\cite{Sz1} showed that for $g\ge2$, $\M_2(N_g)$ can be generated by crosscap slides and by involutions.
In addition, Szepietowski~\cite{Sz2} gave a finite generating set for $\M_2(N_g)$ by crosscap slides for $g\ge3$, which is minimal for $g=3$, $4$.
Following these works, Hirose-Sato~\cite{HS} showed that for $g\ge4$, $H_1(\M_2(N_g);\Z)$ is isomorphic to $\Z_2^{(g-1)^2+\binom{g-1}{3}}$, and gave a minimal generating set for $\M_2(N_g)$ which consists of $(g-1)^2$ crosscap slides and $\binom{g-1}{3}$ squares of Dehn twists about non-separating simple closed curves.
It immediately follows that $\M_2(N_g)$ can be generated by $(g-1)^2+\binom{g-1}{3}$ crosscap slides.
After that, Altun\"oz-Monden-Pamuk-Y{\i}ld{\i}z~\cite{AMPY} gave a minimal generating set for $\M_2(N_g)$ by involutions for $g\ge4$.
Omori and the author showed that $H_1(\T_2(N_g);\Z)$ is isomorphic to $\Z_2^{(g-1)^2+\binom{g-1}{3}-1}$ for $g\ge5$ and to $\Z\oplus\Z_2$ for $g=3$, and gave a finite generating set for $\T_2(N_g)$ which consists of products of two crosscap slides as shown in Figure~\ref{crosscap-slide-rel}~(b) and squares of Dehn twists about non-separating simple closed curves for $g\ge3$, whose cardinality is more than the dimension of $H_1(\T_2(N_g);\Z)$.
Note that $\T_2(N_g)$ can be normally generated by one element as shown in Figure~\ref{crosscap-slide-rel}~(b) (see \cite{KO}) and can not be generated by only squares of Dehn twists (see \cite{IK}).
In this paper, we give a finite generating set for $\T_2(N_g)$ consisting of $(g-1)^2+\binom{g-1}{3}+1$ elements which are described as Dehn twists or their products about explicit simple closed curves, for $g\ge4$ (see Theorem~\ref{gen-1}), and a generating set for $\T_2(N_g)$ consisting of $(g-1)^2+\binom{g-1}{3}+1$ involutions for $g\ge9$ (see Theorem~\ref{gen-2}).

Throughout this paper, the product $gf$ of mapping classes $f$ and $g$ means that we apply $f$ first and then $g$.
In calculations, we use basic relations on $\M(N_g)$.
For details, for instance see \cite{Sz1}.

\section{A small simple generating set for $\T_2(N_g)$}\label{main-1}

For $1\le{i_1}<i_2<\cdots<i_k\le{g}$, let $\alpha_{i_1,i_2,\dots,i_k}$ be an oriented simple closed curve as shown in Figure~\ref{alpha}.
The first main result is as follows.

\begin{figure}[htbp]
\includegraphics{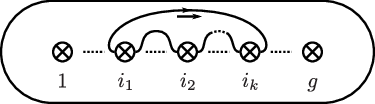}
\caption{An oriented simple closed curve $\alpha_{i_1,i_2,\dots,i_k}$.}\label{alpha}
\end{figure}

\begin{thm}\label{gen-1}
For $g\ge4$, $\T_2(N_g)$ is generated by the following $(g-1)^2+\binom{g-1}{3}+1$ elements.
\begin{itemize}
\item	$t_{\delta_1(\alpha_i,\alpha_{i,j},\alpha_{i,g})}t_{\delta_2(\alpha_i,\alpha_{i,j},\alpha_{i,g})}$ for $1\le{i<j}\le{g-1}$,
\item	$t_{\delta_1(\alpha_i,\alpha_{j,i},\alpha_{i,g})}t_{\delta_2(\alpha_i,\alpha_{j,i},\alpha_{i,g})}$ for $1\le{j<i}\le{g-1}$,
\item	$t_{\alpha_{1,j}}^2$ for $2\le{j}\le{g-1}$,
\item	$t_{\alpha_{1,j,k,l}}^2$ for $2\le{j<k<l}\le{g}$,
\item	$t_{\delta(\alpha_1,\alpha_{1,g})}$ and $t_{\delta(\alpha_{g-1},\alpha_{g-1,g})}$.
\end{itemize}
\end{thm}

For $1\le{i}\le{g}$ and $1\le{j}\le{g}$ with $i\neq{j}$, let
$y_{i,j}=
\left\{
\begin{array}{ll}
Y_{\alpha_i,\alpha_{i,j}}&(i<j),\\
Y_{\alpha_i,\alpha_{j,i}}&(j<i)
\end{array}
\right.$.
Then we have that
\begin{eqnarray*}
&&y_{i,j}y_{i,g}=
\left\{
\begin{array}{ll}
t_{\delta_1(\alpha_i,\alpha_{i,j},\alpha_{i,g})}t_{\delta_2(\alpha_i,\alpha_{i,j},\alpha_{i,g})}&(i<j\le{g-1}),\\
t_{\delta_1(\alpha_i,\alpha_{j,i},\alpha_{i,g})}t_{\delta_2(\alpha_i,\alpha_{j,i},\alpha_{i,g})}&(j<i\le{g-1})
\end{array}
\right.,\\
&&y_{j,i}y_{i,j}^{-1}=y_{j,i}^{-1}y_{i,j}=t_{\alpha_{i,j}}^2~(i<j),\\
&&y_{i,j}^2=t_{\delta(\alpha_i,\alpha_{i,j})}=t_{\delta(\alpha_j,\alpha_{i,j})}=y_{j,i}^2~(i<j).
\end{eqnarray*}
Theorem~\ref{gen-1} can be rephrased as follows.

\begin{thm}
For $g\ge4$, $\T_2(N_g)$ is generated by the following $(g-1)^2+\binom{g-1}{3}+1$ elements.
\begin{itemize}
\item	$y_{i,j}y_{i,g}$ for $1\le{i}\le{g-1}$, $1\le{j}\le{g-1}$ with $i\neq{j}$,
\item	$y_{j,1}y_{1,j}^{-1}$ for $2\le{j}\le{g-1}$,
\item	$t_{\alpha_{1,j,k,l}}^2$ for $2\le{j<k<l}\le{g}$,
\item	$y_{1,g}^2$ and $y_{g-1,g}^2$.
\end{itemize}
\end{thm}

Let $G_1$ be the subgroup of $\T_2(N_g)$ generated by the elements in Theorem~\ref{gen-1}.
Let $G_2$ be the subgroup of $\T_2(N_g)$ generated by the elements in Theorem~\ref{gen-1} and $y_{1,i}$ for $2\le{i}\le{g-1}$ and $y_{2,j}$ for $3\le{j}\le{g}$.
Let $G_3$ be the subgroup of $\T_2(N_g)$ generated by the followings.
\begin{itemize}
\item	$y_{12}y_{i,j}^{\pm1}$ for $1\le{i}\le{g-1}$, $1\le{j}\le{g}$ with $i\neq{j}$,
\item	$t_{\alpha_{1,j,k,l}}^2$ for $2\le{j<k<l}\le{g}$,
\item	$y_{1,2}t_{\alpha_{1,j,k,l}}^2y_{1,2}^{-1}$ for $2\le{j<k<l}\le{g}$,
\end{itemize}
We show $\T_2(N_g)=G_3$, $G_3\subset{G_2}$ and $G_2\subset{G_1}$.
Note that $G_i\subset\T_2(N_g)$ folds for any $i=1$, $2$, $3$.

First, we show $\T_2(N_g)=G_3$.
Hirose-Sato~\cite{HS} proved that for $g\ge4$, $\M_2(N_g)$ is generated by the following $(g-1)^2+\binom{g-1}{3}$ elements.
\begin{itemize}
\item	$y_{i,j}$ for $1\le{i}\le{g-1}$, $1\le{j}\le{g}$ with $i\neq{j}$,
\item	$t_{\alpha_{1,j,k,l}}^2$ for $2\le{j<k<l}\le{g}$.
\end{itemize}
It is known that $\T(N_g)$ (resp. $\T_2(N_g)$) is an index two subgroup of $\M(N_g)$ (resp. $\M_2(N_g)$) (see \cite{Li2} (resp. \cite{KO})).
The quotient $\M(N_g)/\T(N_g)$ and $\M_2(N_g)/\T_2(N_g)$ can be generated by $y_{1,2}$.
We can take a Schreier transversal $U=\{1,y_{1,2}\}$ for $\T_2(N_g)$ in $\M_2(N_g)$.
For $f\in\M_2(N_g)$, we denote by $\overline{f}$ the element of $U$ corresponding to the image of $f$ in $\M_2(N_g)/\T_2(N_g)$.
Let $G$ be the generating set for $\M_2(N_g)$ above.
Then, by the Reidemeister-Schreier method,  $\T_2(N_g)$ is generated by
$$\{uf^{\pm1}\overline{uf^{\pm1}}^{-1}\mid{}u\in{U},f\in{G},uf^{\pm1}\neq\overline{uf^{\pm1}}\}.$$
We see
\begin{eqnarray*}
1y_{i,j}^{\pm1}\overline{1y_{i,j}^{\pm1}}^{-1}&=&y_{i,j}^{\pm1}y_{1,2}^{-1}=(y_{1,2}y_{i,j}^{\mp1})^{-1},\\
y_{1,2}y_{i,j}^{\pm1}\overline{y_{1,2}y_{i,j}^{\pm1}}^{-1}&=&y_{1,2}y_{i,j}^{\pm1},\\
1(t_{\alpha_{1,j,k,l}}^2)^{\pm1}\overline{1(t_{\alpha_{1,j,k,l}}^2)^{\pm1}}^{-1}&=&(t_{\alpha_{1,j,k,l}}^2)^{\pm1},\\
y_{1,2}(t_{\alpha_{1,j,k,l}}^2)^{\pm1}\overline{y_{1,2}(t_{\alpha_{1,j,k,l}}^2)^{\pm1}}^{-1}&=&y_{1,2}(t_{\alpha_{1,j,k,l}}^2)^{\pm1}y_{1,2}^{-1}=(y_{1,2}t_{\alpha_{1,j,k,l}}^2y_{1,2}^{-1})^{-1}.
\end{eqnarray*}
Therefore $\M_2(N_g)=G_3$ folds.

Next, we show $G_3\subset{G_2}$.
For $3\le{i}\le{g-1}$, we see
\begin{eqnarray*}
y_{1,2}y_{i,g}^{-1}&=&y_{i,g}^{-1}y_{1,2}=(y_{i,1}y_{i,g})^{-1}\cdot{}y_{i,1}y_{1,i}^{-1}\cdot{}y_{1,i}y_{1,g}\cdot(y_{1,2}y_{1,g})^{-1}\cdot{}y_{1,2}^2\in{G_2},\\
y_{1,2}y_{i,g}&=&y_{i,g}y_{1,2}=(y_{1,2}y_{i,g}^{-1})^{-1}\cdot{}y_{1,2}^2\in{G_2}.
\end{eqnarray*}
For $3\le{i,j}\le{g-1}$, we see
\begin{eqnarray*}
y_{1,2}y_{i,j}&=&y_{i,j}y_{1,2}=y_{i,j}y_{i,g}\cdot{}y_{1,2}y_{i,g}^{-1}=\in{G_2},\\
y_{1,2}y_{i,j}^{-1}&=&y_{i,j}^{-1}y_{1,2}=(y_{1,2}y_{i,j})^{-1}\cdot{}y_{1,2}^2\in{G_2}.
\end{eqnarray*}
For $3\le{j}\le{g}$, we see
\begin{eqnarray*}
y_{1,2}y_{1,j}^{-1}&=&y_{1,2}y_{1,g}\cdot(y_{1,j}y_{1,g})^{-1}\in{G_2},\\
y_{1,2}y_{1,j}&=&y_{1,2}y_{1,j}^{-1}\cdot{}y_{1,j}^2\in{G_2},\\
y_{1,2}y_{2,j}^{-1}&=&y_{1,2}y_{2,1}^{-1}\cdot{}y_{2,1}y_{2,g}\cdot(y_{2,j}y_{2,g})^{-1}\in{G_2},\\
y_{1,2}y_{2,j}&=&y_{1,2}y_{2,j}^{-1}\cdot{}y_{2,j}^2\in{G_2}.
\end{eqnarray*}
For $3\le{i}\le{g-1}$, we see
\begin{eqnarray*}
y_{1,2}y_{i,1}^{-1}&=&y_{1,2}y_{1,g}\cdot(y_{1,i}y_{1,g})^{-1}\cdot(y_{i,1}y_{1,i}^{-1})^{-1}\in{G_2},\\
y_{1,2}y_{i,1}&=&y_{1,2}y_{i,1}^{-1}\cdot{}y_{i,1}^2=y_{1,2}y_{i,1}^{-1}\cdot{}y_{1,i}^2\in{G_2},\\
y_{1,2}y_{i,2}^{-1}&=&y_{1,2}y_{i,1}^{-1}\cdot{}y_{i,1}y_{i,g}\cdot(y_{i,2}y_{i,g})^{-1}\in{G_2},\\
y_{1,2}y_{i,2}&=&y_{1,2}y_{i,2}^{-1}\cdot{}y_{i,2}^2=y_{1,2}y_{i,2}^{-1}\cdot{}y_{2,i}^2\in{G_2}.
\end{eqnarray*}
Finally, we see
\begin{eqnarray*}
y_{1,2}y_{2,1}^{-1}&=&(y_{2,1}y_{1,2}^{-1})^{-1}\in{G_2},\\
y_{1,2}y_{2,1}&=&(y_{2,1}y_{1,2}^{-1})^{-1}\cdot{}y_{2,1}^2=(y_{2,1}y_{1,2}^{-1})^{-1}\cdot{}y_{1,2}^2\in{G_2}.
\end{eqnarray*}
Therefore $G_3\subset{G_2}$.

Finally, we show $G_2\subset{G_1}$.
We see
\begin{eqnarray*}
y_{1,g-1}^2&=&y_{1,g-1}\cdot{}y_{1,g}y_{g-1,g}\cdot{}y_{1,g-1}^{-1}\cdot{}y_{g-1,g}^{-1}y_{1,g}^{-1}\\
&=&y_{1,g-1}y_{1,g}\cdot{}y_{g-1,g}^2\cdot(y_{g-1,1}y_{g-1,g})^{-1}\cdot{}y_{g-1,1}y_{1,g-1}^{-1}\cdot(y_{g-1,1}y_{g-1,g})^{-1}\\
&&\cdot{}y_{g-1,1}y_{1,g-1}^{-1}\cdot{}y_{1,g-1}y_{1,g}\cdot(y_{1,g}^2)^{-1}\in{G_1}
\end{eqnarray*}
(see Figure~\ref{y}~(a)).
For $2\le{i}\le{g-2}$, we see
$$y_{1,i}^2\cdot{}y_{1,g-1}y_{1,g}\cdot{}y_{i,g-1}y_{i,g}=t_{y_{1,g}y_{i,g}(\alpha_{g-1,g})}^{-1}t_{\alpha_{g-1,g}}=y_{1,g}y_{i,g}t_{\alpha_{g-1,g}}^{-1}y_{i,g}^{-1}y_{1,g}^{-1}t_{\alpha_{g-1,g}}$$
(see Figure~\ref{y}~(b)).
In addition, we see
\begin{eqnarray*}
y_{1,g}y_{i,g}&=&y_{1,g}^2\cdot(y_{1,i}y_{1,g})^{-1}\cdot(y_{i,1}y_{1,i}^{-1})^{-1}\cdot{}y_{i,1}y_{i,g}\in{G_1},\\
t_{\alpha_{g-1,g}}^{-1}y_{i,g}^{-1}y_{1,g}^{-1}t_{\alpha_{g-1,g}}
&=&y_{g-1,g}y_{i,g-1}^{-1}y_{g-1,g}^{-1}\cdot{}y_{g-1,g}y_{1,g-1}^{-1}y_{g-1,g}^{-1}\\
&=&y_{g-1,g}^2\cdot(y_{g-1,1}y_{g-1,g})^{-1}\cdot{}y_{g-1,1}y_{1,g-1}^{-1}\cdot{}y_{1,g-1}y_{1,g}\cdot(y_{1,i}y_{1,g})^{-1}\\
&&\cdot(y_{i,1}y_{1,i}^{-1})^{-1}\cdot{}y_{i,1}y_{i,g}\cdot(y_{i,g-1}y_{i,g})^{-1}\\
&&\cdot(y_{1,g-1}^2)^{-1}\cdot(y_{g-1,1}y_{1,g-1}^{-1})^{-1}\cdot{}y_{g-1,1}y_{g-1,g}\cdot(y_{g-1,g}^2)^{-1}\in{G_1}
\end{eqnarray*}
(see Figure~\ref{y}~(c)).
Hence $y_{1,i}^2$ is in $G_1$ for $2\le{i}\le{g-2}$.
For $3\le{j}\le{g-1}$, we see
\begin{eqnarray*}
y_{2,j}^2
&=&y_{2,j}y_{2,g}\cdot(y_{2,1}y_{2,g})^{-1}\cdot{}y_{2,1}^2\cdot{}y_{2,1}^{-1}y_{2,j}y_{2,1}\cdot{}y_{2,1}^{-1}\\
&=&y_{2,j}y_{2,g}\cdot(y_{2,1}y_{2,g})^{-1}\cdot{}y_{1,2}^2\cdot{}y_{j,1}y_{2,j}^{-1}y_{j,1}^{-1}\cdot{}y_{2,1}^{-1}\\
&=&y_{2,j}y_{2,g}\cdot(y_{2,1}y_{2,g})^{-1}\cdot{}y_{1,2}^2\cdot{}y_{j,1}y_{2,j}^{-1}\cdot{}y_{j,1}^{-1}y_{2,1}^{-1}
\end{eqnarray*}
(see Figure~\ref{y}~(d)).
In addition, we see
\begin{eqnarray*}
y_{j,1}y_{2,j}^{-1}&=&y_{j,1}y_{1,j}^{-1}\cdot{}y_{1,j}y_{1,g}\cdot(y_{1,2}y_{1,g})^{-1}\cdot(y_{2,1}y_{1,2}^{-1})^{-1}\cdot{}y_{2,1}y_{2,g}\cdot(y_{2,j}y_{2,g})^{-1}\in{G_1},\\
y_{j,1}^{-1}y_{2,1}^{-1}&=&(y_{j,1}^2)^{-1}\cdot{}y_{j,1}y_{2,1}^{-1}=(y_{1,j}^2)^{-1}\cdot{}y_{j,1}y_{1,j}^{-1}\cdot{}y_{1,j}y_{1,g}\cdot(y_{1,2}y_{1,g})^{-1}\cdot(y_{2,1}y_{1,2}^{-1})^{-1}\in{G_1}.
\end{eqnarray*}
Hence $y_{2,j}^2$ is in $G_1$ for $3\le{j}\le{g-1}$.
Finally, we see
$$y_{2,g}^2=(y_{2,1}^2y_{2,3}^2\cdots{}y_{2,g-1}^2)^{-1}=(y_{1,2}^2\cdot{}y_{2,3}^2\cdots{}y_{2,g-1}^2)^{-1}\in{G_1}$$
(see Figure~\ref{y}~(e)).
Therefore $G_2\subset{G_1}$.

\begin{figure}[htbp]
\subfigure[$y_{1,g-1}=y_{1,g}y_{g-1,g}\cdot{}y_{1,g-1}^{-1}\cdot{}y_{g-1,g}^{-1}y_{1,g}^{-1}$]{\includegraphics{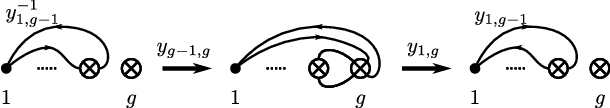}}
\subfigure[$y_{1,i}^2\cdot{}y_{1,g-1}y_{1,g}\cdot{}y_{i,g-1}y_{i,g}=t_{y_{1,g}y_{i,g}(\alpha_{g-1,g})}^{-1}t_{\alpha_{g-1,g}}$]{\includegraphics{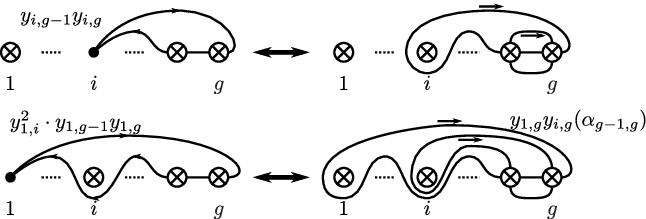}}
\subfigure[$t_{\alpha_{g-1,g}}^{-1}y_{i,g}^{-1}y_{1,g}^{-1}t_{\alpha_{g-1,g}}=y_{g-1,g}y_{i,g-1}^{-1}y_{g-1,g}^{-1}\cdot{}y_{g-1,g}y_{1,g-1}^{-1}y_{g-1,g}^{-1}$]{\includegraphics{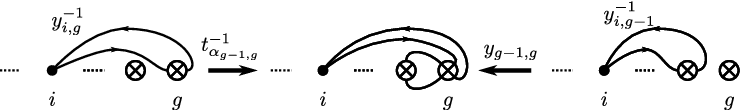}}
\subfigure[$y_{2,1}^{-1}y_{2,j}y_{2,1}=y_{j,1}y_{2,j}^{-1}y_{j,1}^{-1}$]{\includegraphics{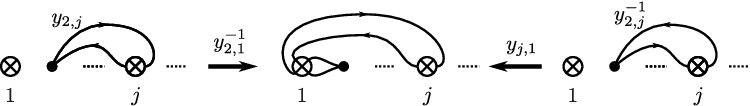}}
\subfigure[$y_{2,g}^2=(y_{2,1}^2y_{2,3}^2\cdots{}y_{2,g-1}^2)^{-1}$]{\includegraphics{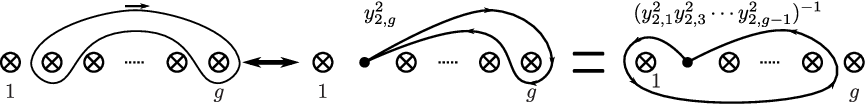}}
\caption{}\label{y}
\end{figure}

Thus, we complete the proof of Theorem~\ref{gen-1}.

\section{A small involution generating set for $\T_2(N_g)$}\label{main-2}

For $1\le{i<j}\le{g}$, let $\beta_{i,j}$ be an oriented simple closed curve as shown in Figure~\ref{beta}.
For $1\le{i<j<k<l}\le{g}$, let $\beta_{i,j,k,l}$ be a simple closed curve as shown in Figure~\ref{beta}.
Let $R$ be the reflection as shown in Figure~\ref{R}.
For $1\le{i}\le{g-1}$, $1\le{j}\le{g}$ with $i\neq{j}$, we put
$$m=
\left\{
\begin{array}{ll}
1&(\min\{i,j\}\ge3~\textrm{or}~(i,j)=(1,2)),\\
\min\{i,j\}+1&(\min\{i,j\}<3,~|i-j|\ge3),\\
\max\{i,j\}+1&(\min\{i,j\}<3,~|i-j|<3,~\textrm{with}~(i,j)\neq(1,2))
\end{array}
\right..$$
For $2\le{j<k<l}\le{g}$, we put
$$n=
\left\{
\begin{array}{ll}
2&(j\ge4),\\
j+1&(j<4,~k-j\ge3),\\
k+1&(j<4,~k-j<3,~l-k\ge3),\\
l+1&(j<4,~k-j<3,~l-k<3)
\end{array}
\right..$$
The second main result is as follows.

\begin{figure}[htbp]
\includegraphics{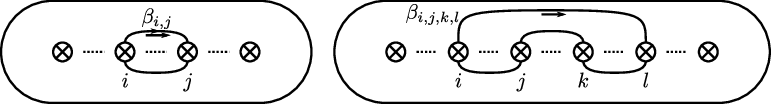}
\caption{An oriented simple closed curve $\beta_{i,j}$ and a simple closed curve $\beta_{i,j,k,l}$.}\label{beta}
\end{figure}

\begin{figure}[htbp]
\includegraphics{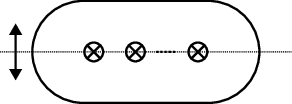}
\caption{The reflection $R$.}\label{R}
\end{figure}

\begin{thm}\label{gen-2}
For $g\ge9$ odd, $\T_2(N_g)$ is generated by the following $(g-1)^2+\binom{g-1}{3}+1$ involutions.
\begin{itemize}
\item	$RY_{\alpha_i,\beta_{i,j}}Y_{\alpha_m,\beta_{m,m+1}}$ for $1\le{i<j}\le{g}$,
\item	$RY_{\alpha_j,\beta_{i,j}}Y_{\alpha_m,\beta_{m,m+1}}$ for $1\le{i<j}\le{g-1}$,
\item	$Rt_{\beta_{1,j,k,l}}t_{R(\beta_{1,j,k,l})}^{-1}$ for $2\le{j<k<l}\le{g}$ and
\item	$R$.
\end{itemize}
For $g\ge10$ even, $\T_2(N_g)$ is generated by the following $(g-1)^2+\binom{g-1}{3}+1$ involutions.
\begin{itemize}
\item	$RY_{\alpha_i,\beta_{i,j}}$ for $1\le{i<j}\le{g}$,
\item	$RY_{\alpha_j,\beta_{i,j}}$ for $1\le{i<j}\le{g-1}$,
\item	$RY_{\alpha_n,\beta_{n,n+1}}t_{\beta_{1,j,k,l}}t_{R(\beta_{1,j,k,l})}^{-1}$ for $2\le{j<k<l}\le{g}$ and
\item	$RY_{\alpha_1,\beta_{1,2}}^{-1}$.
\end{itemize}
\end{thm}

For $1\le{i}\le{g}$ and $1\le{j}\le{g}$ with $i\neq{j}$, let
$z_{i,j}=
\left\{
\begin{array}{ll}
Y_{\alpha_i,\beta_{i,j}}&(i<j),\\
Y_{\alpha_i,\beta_{j,i}}&(j<i)
\end{array}
\right.$.
Theorem~\ref{gen-2} can be rephrased as follows.

\begin{thm}
For $g\ge9$ odd, $\T_2(N_g)$ is generated by the following $(g-1)^2+\binom{g-1}{3}+1$ involutions.
\begin{itemize}
\item	$Rz_{i,j}z_{m,m+1}$ for $1\le{i}\le{g-1}$ and $1\le{j}\le{g}$ with $i\neq{j}$,
\item	$Rt_{\beta_{1,j,k,l}}t_{R(\beta_{1,j,k,l})}^{-1}$ for $2\le{j<k<l}\le{g}$ and
\item	$R$.
\end{itemize}
For $g\ge10$ even, $\T_2(N_g)$ is generated by the following $(g-1)^2+\binom{g-1}{3}+1$ involutions.
\begin{itemize}
\item	$Rz_{i,j}$ for $1\le{i}\le{g-1}$ and $1\le{j}\le{g}$ with $i\neq{j}$,
\item	$Rz_{n,n+1}t_{\beta_{1,j,k,l}}t_{R(\beta_{1,j,k,l})}^{-1}$ for $2\le{j<k<l}\le{g}$ and
\item	$Rz_{1,2}^{-1}$.
\end{itemize}
\end{thm}

\begin{rem}
The reflection $R$ is an involution.
We have the following relations.
\begin{eqnarray*}
Rz_{i,j}^{\pm}R&=&z_{i,j}^{\mp1}\\
Rt_{\beta_{1,j,k,l}}t_{R(\beta_{1,j,k,l})}^{-1}R&=&t_{R(\beta_{1,j,k,l})}t_{\beta_{1,j,k,l}}^{-1}=(t_{\beta_{1,j,k,l}}t_{R(\beta_{1,j,k,l})}^{-1})^{-1},\\
z_{i,j}z_{m,m+1}&=&z_{m,m+1}z_{i,j},\\
z_{n,n+1}t_{\beta_{1,j,k,l}}t_{R(\beta_{1,j,k,l})}^{-1}&=&t_{\beta_{1,j,k,l}}t_{R(\beta_{1,j,k,l})}^{-1}z_{n,n+1}.
\end{eqnarray*}
By these relations, we can check that each element in Theorem~\ref{gen-2} is an involution.
\end{rem}

At first,  we show the following.

\begin{lem}
For $g\ge4$, $\M_2(N_g)$ is minimally generated by the following elements.
\begin{itemize}
\item	$z_{i,j}$ for $1\le{i}\le{g-1}$, $1\le{j}\le{g}$ with $i\neq{j}$,
\item	$t_{\beta_{1,j,k,l}}t_{R(\beta_{1,j,k,l})}^{-1}$ for $2\le{j<k<l}\le{g}$.
\end{itemize}
\end{lem}

\begin{proof}
Let $G$ be the group generated by the elements in the lemma.
Remember the notations $y_{i,j}$'s in Section~\ref{main-2} and that $\M_2(N_g)$ is minimally generated by
\begin{itemize}
\item	$y_{i,j}$ for $1\le{i}\le{g-1}$, $1\le{j}\le{g}$ with $i\neq{j}$,
\item	$t_{\alpha_{1,j,k,l}}^2$ for $2\le{j<k<l}\le{g}$
\end{itemize}
(see~\cite{HS}).
We show that these generators are in $G$.

Let $\mathcal{Y}$ (resp. $\mathcal{Z}$) be the subgroup of $\M_2(N_g)$ generated by $y_{i,j}$ (resp. $z_{i,j}$) for $1\le{i}\le{g-1}$ and $1\le{j}\le{g}$ with $i\neq{j}$.
Altun\"oz-Monden-Pamuk-Y{\i}ld{\i}z~\cite{AMPY} showed $\mathcal{Y}=\mathcal{Z}$.
Hence we check that $t_{\alpha_{1,j,k,l}}^2$ is in $G$ for $2\le{j<k<l}\le{g}$.

For $2\le{j<k<l}\le{g}$, let
$$f_{1,j,k,l}=(z_{2,l}y_{2,k}^{-1}y_{2,j}^{-1}\cdots{}z_{j-1,l}y_{j-1,k}^{-1}y_{j-1,j}^{-1})(z_{j+1,l}y_{j+1,k}^{-1}\cdots{}z_{k-1,l}y_{k-1,k}^{-1})(z_{k+1,l}\cdots{}z_{l-1,l}).$$
Note that $f_{1,j,k,l}$ is in $G$ and that $f_{1,j,k,l}(\beta_{1,j,k,l})=\alpha_{1,j,k,l}$ and $f_{1,j,k,l}(R(\beta_{1,j,k,l}))$ is as shown in Figure~\ref{fR}.
By basic relations on Dehn twists, we have
\begin{eqnarray*}
t_{\alpha_{1,j,k,j}}t_{f_{1,j,k,l}(R(\beta_{1,j,k,l}))}
&=&(t_{\alpha_{1,j}}t_{\alpha_{j,k}}t_{\alpha_{k,l}})^4\\
&=&t_{\alpha_{1,j}}^2\cdot{}t_{\alpha_{j,k}}^2\cdot{}t_{\alpha_{j,k}}^{-1}t_{\alpha_{1,j}}^2t_{\alpha_{j,k}}\cdot{}t_{\alpha_{k,l}}^2\cdot{}t_{\alpha_{k,l}}^{-1}t_{\alpha_{j,k}}^2t_{\alpha_{k,l}}\cdot{}t_{\alpha_{k,l}}^{-1}t_{\alpha_{j,k}}^{-1}t_{\alpha_{1,j}}^2t_{\alpha_{j,k}}t_{\alpha_{k,l}}\\
&=&y_{j,1}y_{1,j}^{-1}\cdot{}y_{k,j}y_{j,k}^{-1}\cdot{}y_{j,k}^{-1}(y_{k,1}y_{1,k}^{-1})y_{j,k}\cdot{}y_{l,k}y_{k,l}^{-1}\\
&&\cdot{}y_{k,l}^{-1}(y_{l,j}y_{j,l}^{-1})y_{k,l}\cdot{}y_{k,l}^{-1}y_{j,l}^{-1}(y_{l,1}y_{1,l}^{-1})y_{j,l}y_{k,l}
\end{eqnarray*}
(see Figure~\ref{chain}).
Note that $y_{g,j}$ is in $\mathcal{Y}$ (see~\cite{Sz2}).
Hence $t_{\alpha_{1,j,k,j}}t_{f_{1,j,k,l}(R(\beta_{1,j,k,l}))}$ is in $\mathcal{Y}=\mathcal{Z}\subset{G}$.
Therefore we have
\begin{eqnarray*}
t_{\alpha_{1,j,k,l}}^2
&=&t_{\alpha_{1,j,k,j}}t_{f_{1,j,k,l}(R(\beta_{1,j,k,l}))}\cdot{}t_{\alpha_{1,j,k,j}}t_{f_{1,j,k,l}(R(\beta_{1,j,k,l}))}^{-1}\\
&=&t_{\alpha_{1,j,k,j}}t_{f_{1,j,k,l}(R(\beta_{1,j,k,l}))}\cdot{}f_{1,j,k,l}(t_{\beta_{1,j,k,l}}t_{R(\beta_{1,j,k,l})}^{-1})f_{1,j,k,l}^{-1}\in{G}.
\end{eqnarray*}

\begin{figure}[htbp]
\includegraphics{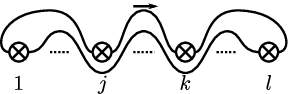}
\caption{A simple closed curve $f_{1,j,k,l}(R(\beta_{1,j,k,l}))$.}\label{fR}
\end{figure}

\begin{figure}[htbp]
\includegraphics{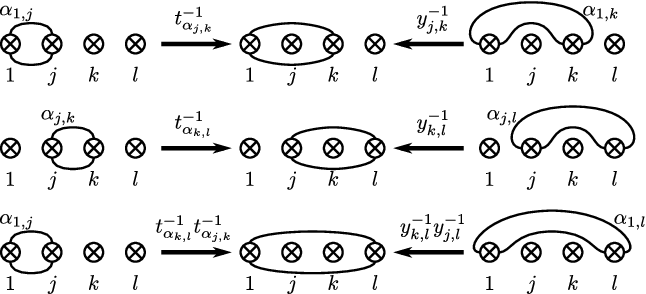}
\caption{$t_{\alpha_{j,k}}^{-1}t_{\alpha_{1,j}}^2t_{\alpha_{j,k}}=y_{j,k}^{-1}(y_{k,1}y_{1,k}^{-1})y_{j,k}$, $t_{\alpha_{k,l}}^{-1}t_{\alpha_{j,k}}^2t_{\alpha_{k,l}}=y_{k,l}^{-1}(y_{l,j}y_{j,l}^{-1})y_{k,l}$, $t_{\alpha_{k,l}}^{-1}t_{\alpha_{j,k}}^{-1}t_{\alpha_{1,j}}^2t_{\alpha_{j,k}}t_{\alpha_{k,l}}=y_{k,l}^{-1}y_{j,l}^{-1}(y_{l,1}y_{1,l}^{-1})y_{j,l}y_{k,l}$}\label{chain}
\end{figure}

Thus we finish the proof.
\end{proof}

For this generating set, we apply the Reidemeister-Schreier method.
Note that the reflection $R$ is describe as $R=z_{1,g}z_{2,g}\cdots{}z_{g-1,g}$ (see \cite{AMPY}).
Hence $R\in\T_2(N_g)$ (resp. $R\notin\T_2(N_g)$) for odd $g$ (resp. even $g$).
Remember that $\T(N_g)$ (resp. $\T_2(N_g)$) is an index two subgroup of $\M(N_g)$ (resp. $\M_2(N_g)$) (see \cite{Li2} (resp. \cite{KO})).
The quotient $\M(N_g)/\T(N_g)$ and $\M_2(N_g)/\T_2(N_g)$ can be generated by $Rz_{1,2}$ (resp. $R$) for odd $g$ (resp. even $g$).
We can take a Schreier transversal $U_o=\{1,Rz_{1,2}\}$ (resp. $U_e=\{1,R\}$) for $\T_2(N_g)$ in $\M_2(N_g)$, for odd $g$ (resp. even $g$).
For $f\in\M_2(N_g)$, we denote by $\overline{f}$ the element of $U$ corresponding to the image of $f$ in $\M_2(N_g)/\T_2(N_g)$.
Let $G$ be the generating set for $\M_2(N_g)$ in the above lemma.

First, we consider the case where $g\ge9$ is odd.
By the Reidemeister-Schreier method,  $\T_2(N_g)$ is generated by
$$\{uf^{\pm1}\overline{uf^{\pm1}}^{-1}\mid{}u\in{U_o},f\in{G},uf^{\pm1}\neq\overline{uf^{\pm1}}\}.$$
We see
\begin{eqnarray*}
1z_{i,j}^{\pm1}\overline{1z_{i,j}^{\pm1}}^{-1}&=&z_{i,j}^{\pm1}(Rz_{1,2})^{-1}=(Rz_{1,2}z_{i,j}^{\mp1})^{-1},\\
Rz_{1,2}z_{i,j}^{\pm1}\overline{Rz_{1,2}z_{i,j}^{\pm1}}^{-1}&=&Rz_{1,2}z_{i,j}^{\pm1},\\
1(t_{\beta_{1,j,k,l}}t_{R(\beta_{1,j,k,l})}^{-1})^{\pm1}\overline{1(t_{\beta_{1,j,k,l}}t_{R(\beta_{1,j,k,l})}^{-1})^{\pm1}}^{-1}&=&(t_{\beta_{1,j,k,l}}t_{R(\beta_{1,j,k,l})}^{-1})^{\pm1},\\
Rz_{1,2}(t_{\beta_{1,j,k,l}}t_{R(\beta_{1,j,k,l})}^{-1})^{\pm1}\overline{Rz_{1,2}(t_{\beta_{1,j,k,l}}t_{R(\beta_{1,j,k,l})}^{-1})^{\pm1}}^{-1}&=&(Rz_{1,2}t_{\beta_{1,j,k,l}}t_{R(\beta_{1,j,k,l})}^{-1}Rz_{1,2})^{\pm1}.
\end{eqnarray*}
Note that, by the assumption $g\ge9$,
\begin{eqnarray*}
Rz_{1,2}z_{i,j}^{-1}&=&Rz_{1,2}^2\cdot{}z_{1,2}^{-1}z_{i,j}^{-1}=Rz_{1,2}^2\cdot{}Rz_{1,2}z_{i,j}\cdot{}R,\\
Rz_{1,2}t_{\beta_{1,j,k,l}}t_{R(\beta_{1,j,k,l})}^{-1}Rz_{1,2}&=&Rz_{1,2}z_{n,n+1}\cdot{}t_{\beta_{1,j,k,l}}t_{R(\beta_{1,j,k,l})}^{-1}\cdot{}Rz_{1,2}z_{n,n+1}.
\end{eqnarray*}
Hence $\T_2(N_g)$ is generated by
\begin{itemize}
\item	$Rz_{1,2}z_{i,j}$ for $1\le{i}\le{g-1}$, $1\le{j}\le{g}$ with $i\neq{j}$,
\item	$t_{\beta_{1,j,k,l}}t_{R(\beta_{1,j,k,l})}^{-1}$ for $2\le{j<k<l}\le{g}$ and
\item	$R$.
\end{itemize}
Note that the generator $t_{\beta_{1,j,k,l}}t_{R(\beta_{1,j,k,l})}^{-1}$ is not an involution and the generator $Rz_{1,2}z_{i,j}$ is not necessarily an involution.
When $m=1$, we see
\begin{eqnarray*}
Rz_{1,2}z_{i,j}=Rz_{i,j}z_{1,2}=Rz_{i,j}z_{m,m+1}\in{G}.
\end{eqnarray*}
When $m\ge3$, we see
\begin{eqnarray*}
Rz_{1,2}z_{i,j}=Rz_{1,2}^2\cdot{}Rz_{m,m+1}z_{1,2}\cdot{}Rz_{i,j}z_{m,m+1}\in{G}.
\end{eqnarray*}
When $m=2$, we see
\begin{eqnarray*}
Rz_{1,2}z_{i,j}=Rz_{4,5}z_{1,2}\cdot{}Rz_{m,m+1}z_{4,5}\cdot{}Rz_{i,j}z_{m,m+1}\in{G}.
\end{eqnarray*}
Finally, we see
\begin{eqnarray*}
t_{\beta_{1,j,k,l}}t_{R(\beta_{1,j,k,l})}^{-1}=R\cdot{}Rt_{\beta_{1,j,k,l}}t_{R(\beta_{1,j,k,l})}^{-1}\in{G}.
\end{eqnarray*}
Therefore $\M_2(N_g)=G$ folds.

Next, we consider the case where $g\ge10$ is even.
By the Reidemeister-Schreier method,  $\T_2(N_g)$ is generated by
$$\{uf^{\pm1}\overline{uf^{\pm1}}^{-1}\mid{}u\in{U_e},f\in{G},uf^{\pm1}\neq\overline{uf^{\pm1}}\}.$$
We see
\begin{eqnarray*}
1z_{i,j}^{\pm1}\overline{1z_{i,j}^{\pm1}}^{-1}&=&z_{i,j}^{\pm1}R^{-1}=(Rz_{i,j}^{\mp1})^{-1},\\
Rz_{i,j}^{\pm1}\overline{Rz_{i,j}^{\pm1}}^{-1}&=&Rz_{i,j}^{\pm1},\\
1(t_{\beta_{1,j,k,l}}t_{R(\beta_{1,j,k,l})}^{-1})^{\pm1}\overline{1(t_{\beta_{1,j,k,l}}t_{R(\beta_{1,j,k,l})}^{-1})^{\pm1}}^{-1}&=&(t_{\beta_{1,j,k,l}}t_{R(\beta_{1,j,k,l})}^{-1})^{\pm1},\\
R(t_{\beta_{1,j,k,l}}t_{R(\beta_{1,j,k,l})}^{-1})^{\pm1}\overline{R(t_{\beta_{1,j,k,l}}t_{R(\beta_{1,j,k,l})}^{-1})^{\pm1}}^{-1}&=&(Rt_{\beta_{1,j,k,l}}t_{R(\beta_{1,j,k,l})}^{-1}R^{-1})^{\pm1}.
\end{eqnarray*}
Note that
$$Rt_{\beta_{1,j,k,l}}t_{R(\beta_{1,j,k,l})}^{-1}R^{-1}
=t_{R(\beta_{1,j,k,l})}t_{\beta_{1,j,k,l}}^{-1}
=(t_{\beta_{1,j,k,l}}t_{R(\beta_{1,j,k,l})}^{-1})^{-1}.$$
Hence $\T_2(N_g)$ is generated by
\begin{itemize}
\item	$Rz_{i,j}^{\pm1}$ for $1\le{i}\le{g-1}$, $1\le{j}\le{g}$ with $i\neq{j}$,
\item	$t_{\beta_{1,j,k,l}}t_{R(\beta_{1,j,k,l})}^{-1}$ for $2\le{j<k<l}\le{g}$.
\end{itemize}
Note that the generator $t_{\beta_{1,j,k,l}}t_{R(\beta_{1,j,k,l})}^{-1}$ is not involution.
When $i$, $j\ge3$, we see
\begin{eqnarray*}
Rz_{i,j}^{-1}=z_{i,j}R\cdot{}z_{1,2}^{-1}z_{1,2}=z_{1,2}z_{i,j}Rz_{1,2}=Rz_{1,2}^{-1}\cdot{}Rz_{i,j}\cdot{}Rz_{1,2}\in{G}.
\end{eqnarray*}
When $\min\{i,j\}\le2$ and $\max\{i,j\}\ge5$, we see
\begin{eqnarray*}
Rz_{i,j}^{-1}=z_{i,j}R\cdot{}z_{3,4}^{-1}z_{3,4}=z_{3,4}z_{i,j}Rz_{3,4}=Rz_{3,4}^{-1}\cdot{}Rz_{i,j}\cdot{}Rz_{3,4}\in{G}.
\end{eqnarray*}
When $\min\{i,j\}\le2$ and $\max\{i,j\}\le4$, we see
\begin{eqnarray*}
Rz_{i,j}^{-1}=z_{i,j}R\cdot{}z_{5,6}^{-1}z_{5,6}=z_{5,6}z_{i,j}Rz_{5,6}=Rz_{5,6}^{-1}\cdot{}Rz_{i,j}\cdot{}Rz_{5,6}\in{G}.
\end{eqnarray*}
Finally, we see
\begin{eqnarray*}
t_{\beta_{1,j,k,l}}t_{R(\beta_{1,j,k,l})}^{-1}=Rz_{n,n+1}\cdot{}Rz_{n,n+1}t_{\beta_{1,j,k,l}}t_{R(\beta_{1,j,k,l})}^{-1}\in{G}.
\end{eqnarray*}
Therefore $\M_2(N_g)=G$ folds.

Thus, we complete the proof of Theorem~\ref{gen-2}.

\appendix
\section{}

For a finitely generated group $G$, let $m(G)$ be the minimal number of generators for $G$ and, if $G$ can be generated by involutions, $im(G)$ the minimal number of involution generators for $G$.
By~\cite{Sz1,HS,AMPY}, we have
$$m(\M_2(N_g))=im(\M_2(N_g))=(g-1)^2+\binom{g-1}{3},~~m(\M_2(N_3))=4$$
for $g\ge4$.
By~\cite{KO} and our results, we have
$$(g-1)^2+\binom{g-1}{3}-1\leq{m(\T_2(N_g))}\le(g-1)^2+\binom{g-1}{3}+1,$$
$$m(\T_2(N_4))\le11,~~2\le{m(\T_2(N_3))}\le3$$
for $g\ge5$ and
$$(g-1)^2+\binom{g-1}{3}-1\leq{im(\T_2(N_g))}\le(g-1)^2+\binom{g-1}{3}+1$$
for $g\ge9$.

We have the following natural problem.

\begin{prob}
Determine the numbers $m(\T_2(N_g))$ and $im(\T_2(N_g))$ for $g\ge4$.
\end{prob}

For $g\geq5$, we see
\begin{eqnarray*}
t_{\delta(\alpha_1,\alpha_{1,g})}&=&y_{1,g}^2\\
&=&y_{1,g}\cdot(y_{g,2}y_{1,2})y_{1,g}^{-1}(y_{g,2}y_{1,2})^{-1}\cdot{}y_{3,4}y_{3,4}^{-1}\\
&=&y_{1,g}y_{3,4}\cdot{}y_{g,2}y_{1,2}\cdot(y_{1,g}y_{3,4})^{-1}\cdot(y_{g,2}y_{1,2})^{-1}
\end{eqnarray*}
(see Figure~\ref{y_{1,g}}).
Hence $t_{\delta(\alpha_1,\alpha_{1,g})}$, similarly $t_{\delta(\alpha_{g-1},\alpha_{g-1,g})}$, are in the commutator subgroup of $\T_2(N_g)$, and so in the kernel of the natural homomorphism $\T_2(N_g)\to{}H_1(\T_2(N_g),\Z)\cong\Z_2^{(g-1)^2+\binom{g-1}{3}-1}$.
Therefore if $\T_2(N_g)$ can be generated by $(g-1)^2+\binom{g-1}{3}-1$ elements of the generating set in Theorem~\ref{gen-1}, then generators $t_{\delta(\alpha_1,\alpha_{1,g})}$ and $t_{\delta(\alpha_{g-1},\alpha_{g-1,g})}$ are not needed.

\begin{figure}[htbp]
\includegraphics{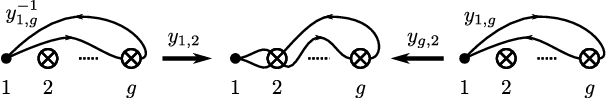}
\caption{$y_{1,g}=(y_{g,2}y_{1,2})y_{1,g}^{-1}(y_{g,2}y_{1,2})^{-1}$}\label{y_{1,g}}
\end{figure}


\end{document}